\documentclass[a4paper,12pt]{article}

\usepackage[usenames,dvipsnames]{xcolor}
\usepackage{amssymb}
\usepackage{amsfonts}
\usepackage{tabulary}
\usepackage{amsmath}
\usepackage{textcomp}
\usepackage{setspace}
\usepackage{graphicx}
\usepackage{authblk}
\newtheorem{theorem}{Theorem}
\newtheorem{definition}[theorem]{Definition}

\newtheorem{proposition}[theorem]{Proposition}
\newtheorem{corollary}[theorem]{Corollary}
\newtheorem{remark}[theorem]{Remark}

\newenvironment{proof}[1][Proof]{\noindent\textbf{#1. }}{\ $\square$\bigskip}

\begin{document}

\title{Fundamental group of  compact decent spacetimes with lightlike hypersurface curvature}

\author{Raymond Hounnonkpe \and Siba Kalivogui}
\maketitle

\vskip 0.2truecm

\begin{abstract}
In this paper, we study the fundamental group of  compact decent $4$-dimensional spacetimes with lightlike hypersurface curvature. Such manifolds are foliated by totally geodesic null hypersurfaces. We prove that when the leaves of the foliation are compact then the fundamental group of the spacetime is polycyclic. We also prove that  a non totally vicious $4$-dimensional compact pp-wave spacetime has polycyclic fundamental group.
\end{abstract}

\begingroup
\renewcommand\thefootnote{}
\endgroup

\section{Introduction}
The notion of fundamental group of a manifold is  one of the important concept in geometry. For instance, in Riemannian geometry, many important results  show that curvature imposes strong restrictions on the algebraic properties of this group (Synge's theorem, Bonnet-Myers's theorem, Preissman's theorem,...). In Lorentzian geometry, curvature assumption also resticts the fundamental group. For instance, in  \cite{Gold}, it is proved that the fundamental group of a compact flat Lorentzian manifold  is virtually polycyclic. 

An important class of Lorentzian manifolds is the class of manifolds with a parallel rank one  lighlike distribution. This class includes Brinkmann spacetimes, pp-waves and plane waves spacetimes which  within in the context of general relativity  occur as  important solutions of the Einstein field equations. 

In the case of plane waves spacetimes,  their fundamental group where investigated in \cite{Zeg}. The aim of this paper is to contribute to the investgation on the properties of the fundamental group of decent spacetimes (see section 2 for the definition) which include  Brinkmann spacetimes, pp-waves and plane waves spacetimes. The organization of the paper is as follows. Section 2 is devoted to some preliminaries on decent spacetimes. In section 3, we prove that the fundamental group of a  compact decent $4$-dimensional spacetimes with lightlike hypersurface curvature  is polycyclic if the codimension one foliation induced by the parallel lighltike distribution has compact leaves (Theorem \ref{theorem pri}). We also deduce that a non totally vicious $4$-dimensional compact pp-wave spacetime has polycyclic fundamental group (corollary \ref{cor pp}).

\section{Preliminaries on decent spacetimes}
In this section, we recall some basics on decent spacetimes (see \cite{Larz}, for details).

Let $(M,g)$ be a  spacetime ( time oriented Lorentzian manifold). We denote by $\nabla$ its Levi-Civita connection.
\begin{definition}
 The spacetime $(M,g)$ is called almost decent if there exists a lightlike vector field $V$ such that 
 \begin{equation}
 \label{decent}
 \nabla_UV = \beta(U)V, \; \forall \; U \in  \Gamma(TM).
 \end{equation}
 \end{definition}
 In this case, the rank one distribution $\mathbb{L}$  generated by $V$ is parallel. The orthogonal distribution $\mathbb{L}^\perp$ ( which cantains $\mathbb{L}$ ) is integrable and so defines a codimension one foliation which we denote by $\mathcal{L}^\perp$. The leaves of $\mathcal{L}^\perp$ are totally geodesic lightlike hypersurfaces of $M$.
 
The spacetime $(M,g)$ is called decent if it almost decent and satisfies  $$Ker \beta = \mathbb{L}^\perp. $$ In this case, it is denoted by $(M, g, V)$.

A decent spacetime is a Brinkmann spacetime if the $1$-form $\beta$ is identically zero ($\beta = 0$), i.e, $V$ is a lighlike parallel vector field.

For a decent spacetime, the $1$-form $\alpha = g(V, .)$ metrically equivalent to $V$ is closed (see \cite[Lemma 2.47]{Larz}). Since $Ker \alpha = \mathbb{L}^\perp$, it follows that the codimension one foliation $\mathcal{L}^\perp$ is defined by a closed $1$-form. If we assume $M$ to be compact then all the leaves are diffeomorphic. Moreoover,  either all the leaves of $\mathcal{L}^\perp$ are compact or they are all dense (\cite[ page 237]{God}).

\begin{definition}
A decent spacetime $(M,g)$ is said to have lightlike hypersurface curvature, if and only if the curvature $R$ satisfies
\begin{equation}
\label{hyper curv}
R(X,Y)Z  \in \Gamma (\mathbb{L}), \; \forall \; X, Y, Z \in \Gamma (\mathbb{L}^\perp).
\end{equation}
\end{definition}

 A spacetime $(M,g)$ is called a pp-wave if it admits a global
parallel lightlike vector field and if its curvature tensor $R$ satisfies
$$R(X,Y) = 0, \; \forall \; X, Y \in \Gamma (\mathbb{L}^\perp).$$

It follows from the above definitions that a pp-wave spacetime is a decent spacetime with lightlike hypersurface curvature.

The following proposition will be uselful fors us (see \cite[Corollary 2.50]{Larz}).
\begin{proposition}
\label{prop decent imp}
Let $(M,g, V)$ be a compact decent spacetime. If the foliation $\mathcal{L}^\perp$ has compact leaves and $F$ is one of its leaves, then the inclusion $F \rightarrow M$ induces a monomorphism $\pi_1(F) \rightarrow \pi_1(M)$ onto a normal subgroup $H$ such  that $\pi_1(M)/ \pi_1(F) = \mathbb{Z}$.
\end{proposition}
\begin{remark}
\label{remark commut}
From, \cite[page 45]{God}, the normal subgroup $H$ which we identify with $\pi_1(F)$
contains the subgroup of the commutators.
\end{remark}

\section{Main results}
We recall some definitions from group theory (see \cite{Rose} ).
\begin{definition}
Let $G$ be a group with identity element $e$. We define subgroups $G^{(n)}$ of $G$, one for each non negative integer $n$, recursively by 
$$ G^{(0)} = G$$
and for $n > 0$, 
$$ G^{(n)} = [G^{(n-1)} , G^{(n-1)}] = (G^{(n-1)})'.$$
By definition 
$$ G=G^{(0)} \geq G^{(1)} \geq G^{(2)}\geq ... $$
This descending sequence of subgroups of $G$ is called the derived series of $G$.

A group $G$ is solvable if there exists $n\in \mathbb{N}$ such that $G^{(n)} = \{ e\}.$

Equivalently, $G$ is solvable if there exists a series
$$ H_0 = \{e\} \scalebox{1.3}{$\triangleleft$}\, H_1 \scalebox{1.3}{$\triangleleft$}\, H_2 \scalebox{1.3}{$\triangleleft$}\, ...\scalebox{1.3}{$\triangleleft$}\, H_n = G$$ with abelian factors $H_i/H_{i-1}, (i= 1,2,...,n).$

A group $G$ is said to be polycyclic if it has a series all of whose factors are cyclic.

It follows that every polycyclic group is solvable.
\end{definition}
Polycyclic group are very  important and appear in the study of compact flat spacetimes. For instance, in \cite{Gold}, it is proved that the fundamental group of a compact flat Lorentzian manifold  is virtually polycyclic, i.e, it has a polycyclic subgroup of finite index.

We have the following.
\begin{theorem} (\cite[Theorem7.47 and Page 156]{Rose})\\
\label{polyc}
Let $K$ be a normal subgroup of $G$. 
\begin{enumerate}
\item If $K$ and $G/K$ are solvable then $G$ is solvable.
\item If $K$ and $G/K$ are polycyclic then $G$ is polycyclic.
\end{enumerate}
\end{theorem}
Now we state the main theorem of the paper. 

\begin{theorem}
\label{theorem pri}
Let $(M,g, V)$ be a $4$-dimensional compact decent spacetime with lightlike hypersurface curvature. If the leaves of the foliation  $\mathcal{L}^\perp$ are compact then the fundamental group $\pi_1(M)$ is polycyclic.
\end{theorem}
\begin{proof}
Let $F$ be a leaf of the codimension one foliation  $\mathcal{L}^\perp$. Then $F$ is a compact totally geodesic null hypersurface (see section 2). From \cite{RM, GR}, the lightlike vector field $V$ which is tangent to $F$  defines a Riemannian flow on $F$. From (\ref{hyper curv}),  the sectional curvature of any non degenerate plane $(X, Y) $ with $X, Y,  \in \Gamma (\mathbb{L}^\perp)$ is zero, i.e, $K(X,Y) = 0$.  It follows form \cite[Proposition 3.2]{GR} that the transverse curvature of the Riemannian flow defined by $V$ on $F$ is zero, i.e, the Riemannian flow is transversally euclidean (see \cite[Definition 3.3]{GR} ). From \cite[Proposition 3.3]{GR} (see also \cite[Proposition 5.7]{Blu}), $\pi_1(F)$ is solvable. Now it is proved in \cite[Theorem 5.2]{Evans}, that if the fundamental group of a compact $3$-manifold is solvable then it is polycyclic. Hence $\pi_1(F)$ is polycyclic. From Proposition \ref{prop decent imp}, $\pi_1(M)/ \pi_1(F) = \mathbb{Z}$. It follows that $\pi_1(M)/ \pi_1(F)$ is  polycyclic. Since $\pi_1(F)$  and $\pi_1(M)/ \pi_1(F)$ are polycyclic, we conclude that $\pi_1(M)$ is polycyclic.
\end{proof}

\begin{remark}
 Suppose $(M,g, V)$ is a compact decent manifold (not necessarly with lightlike hypersurface curvature) such that a leaf $F$ of $\mathcal{L}^\perp$ is compact. From Proposition \ref{prop decent imp} and  Remark \ref{remark commut}, if $\pi_1(F)$  is abelian, then the derived subgroup $G'$ is abelian hence $G'' = \{e\}.$ So $G$ is solvable of length at most $2$ hence $G$ is meta-abelian.
\end{remark}

A spacetime $(M,g)$ is totally vicious if there passes a closed timelike curve through every point $p$ of $M$. 
\begin{corollary}
\label{cor pp}
Let $(M,g)$ be a non totally vicious $4$-dimensional compact pp-wave spacetime with parallel lightlike vector field $V$. Then    
its fundamental group $\pi_1(M)$ is polycyclic.
\end{corollary}
\begin{proof}
From \cite[Theorem 3.4]{AGR} (see also \cite[proof of Theorem 4]{Ray} ), the leaves of the foliation  $\mathcal{L}^\perp$ are compact. The conclusion follows directly from Theorem \ref{theorem pri}.
\end{proof}

\bibliographystyle{amsplain}

\bigskip

\noindent\textbf{Raymond Hounnonkpe}\\
Institut Ouest Africain de Mathématiques, Université Gamal Abdel Nasser de Conakry, Guinée.\\
E-mail: \texttt{adivignon.hounnonkpe@ioam.uganc.edu.gn}\\
E-mail: \texttt{rhounnonkpe@ymail.com}

\medskip

\noindent\textbf{Siba Kalivogui}\\
Université de Kindia, Guinée.\\
E-mail: \texttt{sibakalivogui66@gmail.com }

\end{document}